\documentclass[11pt,reqno]{amsart}
\usepackage{amssymb,amscd,mathrsfs}
\usepackage[colorlinks=true,linkcolor=black,citecolor=black]{hyperref}

\usepackage[usenames,dvipsnames]{color}

\newtheorem{thm}{Theorem}[section]
\newtheorem{cor}[thm]{Corollary}
\newtheorem{lem}[thm]{Lemma}
\newtheorem{prop}[thm]{Proposition}

\theoremstyle{remark}
\newtheorem{rem}[thm]{Remark}
\theoremstyle{definition}
\newtheorem{defn}[thm]{Definition}
\newtheorem{ex}[thm]{Example}

\numberwithin{equation}{section}

\newcommand{\inn}[3]{{_{#1}}{\langle #2 \rangle}{_{#3}}}
\newcommand{\hs}[2]{{\langle #1, #2 \rangle}} 

\newcommand{\mv}{\underline}

\newcommand{\A}{\mathcal{A}}
\newcommand{\B}{\mathcal{B}}
\newcommand{\cL}{\mathcal{L}}
\newcommand{\cK}{\mathcal{K}}
\newcommand{\cO}{\mathcal{O}}
\newcommand{\cH}{\mathcal{H}}

\newcommand{\CP}{\mathbb{CP}}
\newcommand{\WP}{\mathbb{WP}}

\newcommand{\IN}{\mathbb{N}}
\newcommand{\IZ}{\mathbb{Z}}
\newcommand{\IC}{\mathbb{C}}
\newcommand{\IT}{\mathbb{T}}
\newcommand{\IS}{\mathbb{S}}

\newcommand{\Cs}{$C^*$-}

\newcommand{\Span}{\mathrm{Span}}

\newcommand{\Bnd}{\cL(E)}
\newcommand{\BndB}{\cL_B(E)}

\newcommand{\K}{\mathcal{K}(E)}
\newcommand{\KB}{\mathcal{K}_B(E_B)}
\newcommand{\KA}{\mathcal{_A K}(_AE)}

\DeclareMathOperator{\id}{id}

\newcommand{\hot}{\,\widehat{\otimes}}
\newcommand{\ots}{\otimes}
\newcommand{\wt}{\widetilde}
\newcommand{\wh}{\widehat}
\newcommand{\ii}{\rm i}

\begin{document}

%
%
%
%
%
%
%
%
%

\title[]{Pimsner algebras and circle bundles}
\author{Francesca Arici} 
\address{SISSA, Via Bonomea 265, 34136 Trieste, Italy
and I.N.F.N. Sezione di Trieste, Trieste, Italy.}
\email{farici@sissa.it}

\author{Francesco D'Andrea} 
\address{Universit\`a di Napoli ``Federico II'' and I.N.F.N. Sezione di Napoli, Complesso MSA, Via Cintia, 80126 Napoli, Italy.}
\email{francesco.dandrea@unina.it}

\author{Giovanni Landi}
\address{Universit\`{a} di Trieste, Via A.~Valerio~12/1, 34127 Trieste, Italy and I.N.F.N. Sezione di Trieste, Trieste, Italy.}
\email{landi@units.it}



\date{v1 June 2015 ; v2 October 2015}


\begin{abstract}
We report on the connections between noncommutative principal circle bundles, Pimsner algebras and strongly graded algebras. We illustrate several results with the examples of quantum weighted projective and lens spaces and 
$\theta$-deformations.
\end{abstract}

\maketitle

\tableofcontents

\newpage
\parskip 3pt

\section{Introduction}

%

This paper is devoted to Pimsner (or Cuntz-Krieger-Pimsner) algebras, focusing on their connections with noncommutative principal circle bundles as well as with (strongly) graded algebras. 

Pimsner algebras, which were introduced in the seminal work \cite{P97}, provide a unifying framework for a range of important $C^*$-algebras including crossed products by the integers, Cuntz-Krieger algebras \cite{C77, CK80}, and $C^*$-algebras associated to partial automorphisms \cite{E94}. 
Due to their flexibility and wide range of applicability, 
there has been recently an increasing interest in these algebras (see for instance \cite{GG13, RRS15}).
A  related class of algebras, known as generalized crossed products, was independently invented in \cite{AEE}. The two notions coincide in many cases, in particular in those of interest for the present paper. 
Here we will use a more geometrical point of view, showing how certain Pimsner algebras, coming 
from a \emph{self-Morita equivalence bimodule} over a \mbox{\Cs algebra}, can be thought of as algebras of functions on the total space of a noncomutative principal circle bundle, along the lines of \cite{AKL14,DAL14}.

Classically, starting from a principal circle bundle $P$ over a compact topological space $X$, 
an important role is in the associated line bundles. Given any of these,
the corresponding module of sections is a self-Morita equivalence bimodule for the commutative $C^*$-algebra $C(X)$ of continuos functions over $X$. Suitable tensor powers of the (sections of the) bundle are endowed with an algebra structure eventually giving back the $C^*$-algebra $C(P)$ of continuos functions over $P$. This is just a Pimsner algebra construction. By analogy then, one thinks of a self-Morita equivalence bimodule over an arbitrary \mbox{\Cs algebra} as a \emph{noncommutative line bundle} and of the corresponding Pimsner algebra as the `total space' algebra of a principal circle fibration. 

The Euler class of a (classical) line bundle has an important use in the Gysin sequence in complex K-theory, that relates the topology of the base space $X$ to that of the total space $P$ of the bundle. This sequence has natural counterparts in the context of Pimsner algebras, counterparts given by two sequences in KK-theory with natural analogues of the Euler class and a central role played by index maps of canonical classes. 

In order to make this review self-contained, we start in \S\ref{se:1} from recalling the theory of Hilbert modules and Morita equivalences, focusing on those definitions and results that will be needed in the sequel of the paper. Then 
\S\ref{s:2} is  devoted to Pimsner's construction and to the construction of generalized crossed products. This is followed by the six-term exact sequences in KK-theory. We next move in \S\ref{se:3} to noncommutative principal circle bundles and graded algebras and recall how principality of the action can be translated into an algebraic condition on the induced grading. This condition is particularly relevant and it resembles a similar condition appearing in the theory of generalized crossed products.
We then show how all these notions are interconnected and can be seen as different aspects of the same phenomenon. Finally, \S\ref{se:4} is devoted to examples: we illustrate how theta deformed and quantum weighted projective and lens spaces fit into the framework.

\section{Hilbert \mbox{\Cs modules} and Morita equivalence}\label{se:1}
Hilbert \mbox{\Cs modules} are crucially important in modern developments of noncommutative geometry and index theory. They are a generalization of Hilbert spaces where the complex scalars are replaced by a \mbox{\Cs algebra}. From a geometrical point of view, they can be thought of as modules of sections of a noncommutative Hermitian vector bundle. 

We recall 
here some of the definitions and results that we need later on in the paper. 
Our main references for this section are \cite{La95,RW}. 

\subsection{Hilbert \Cs modules}

\begin{defn}
Let $B$ be a  \mbox{\Cs algebra}. A right \emph{pre-Hilbert \Cs module} over $B$ is a right $B$-module $E$ with a $B$-valued Hermitian product 
 $\inn{}{\cdot,\cdot}{B}: E\times E \rightarrow B$ satisfying, for all $\xi, \eta \in E$ and for all $b \in B$, the conditions:
\begin{align*}
& \inn{}{\xi,\eta}{B} =\big( \inn{}{ \eta,\xi }{B} \big)^* \, , \\
& \inn{}{\xi, \eta b}{B} = \inn{}{\xi,\eta}{B} b \, , \\
& \inn{}{\xi,\xi}{B} \geq 0 \, , \quad \mbox{and} \quad 
 \inn{}{\xi, \xi}{B} = 0 \iff \xi = 0 .
 \end{align*}
A  right \emph{Hilbert \Cs module} $E$ over $B$ is a right pre-Hilbert \mbox{\Cs module} which is complete in the norm $\| \cdot \|$ on $E$ defined, via the norm $\| \cdot \|_B$ on $B$, by
$$
 \| \xi \|^2 = \| \inn{}{\xi,\xi}{B} \|_B \; .
$$
Finally, it is said to be \emph{full} if the ideal 
$$
  \inn{}{E,E}{B} := \Span_B \lbrace \inn{}{\xi, \eta}{B} \, \vert \, \xi, \eta \in E \rbrace
$$
is dense in $B$. 
\end{defn}

\noindent
There are analogous definitions for \emph{left} modules, with Hermitian product denoted $\inn{B}{\,,\,}{}$; this is taken to be $B$- linear in the first entry, 

To lighten notations we shall write $\inn{}{\cdot,\cdot}{}=\inn{}{\cdot,\cdot}{B}$ whenever possible and use the name
Hilbert $B$-module to mean a Hilbert \mbox{\Cs module} over the \mbox{\Cs algebra} $B$.
The simplest example of a Hilbert \mbox{$B$-module} is the algebra $B$ itself with respect to the Hermitian product
$$
\inn{}{a,b}{}=a^* b \; .
$$

The module $B^n$ consists of $n$-tuples of elements of $B$, with component-wise operations, and with Hermitian product defined by 
 \begin{equation}
 \label{eq:Bn}
 \inn{}{(a_1, \dots, a_n),(b_1, \dots b_n)}{} = \sum\nolimits_{i=1}^n \inn{}{a_i, b_i}{} \; .
 \end{equation}

Generalizing $B^n$, whenever one has a finite set $\lbrace E_i \rbrace_{i=1}^{n}$ of Hilbert \mbox{$B$-modules}, one can form the direct sum $\bigoplus_{i=1}^{n} E_i$. It is a Hilbert $B$-module in a obvious way (component-wise) with Hermitian product defined as in \eqref{eq:Bn}.

Things become subtler if $\lbrace E_i \rbrace_{i \in I} $ is an infinite collection of Hilbert \mbox{$B$-modules}. Indeed, one needs to replace $\bigoplus_{i \in I} E_i$ by the set of sequences $(\xi_i)$, with $\xi_i \in E_i$ and such that $\sum_{i \in I} \langle \xi_i , \xi_i \rangle$ converges in $B$.
Then for $\xi=(\xi_i)$ and $\eta=(\eta_i)$,
one gets a complete module (hence a Hilbert $B$-module) with the natural inner product:
$$
\langle \xi, \eta \rangle =  \sum\nolimits_{i \in I} \langle \xi_i, \eta_i \rangle
$$

A Hilbert $B$-module $E$ is \emph{finitely generated} if there exists a finite collection $\lbrace \eta_1, \dots, \eta_n \rbrace$ of elements of $E$ such that the $B$-linear span of the $\eta_i$'s is dense in $E$. It is \emph{algebraically} finitely generated if every element $\xi \in E$ is of the form 
$\xi = \sum_{j=1}^n \eta_j  b_j$ for some $b_j$'s in $B$.

A Hilbert $B$-module $E$ is \emph{projective} if it is a direct summand in $B^n$ for some $n$. By \cite[15.4.8]{WO}, 
every algebraically finitely generated Hilbert \mbox{$C^*$-module} over a unital $C^*$-algebra is projective. 

Next, one defines operators between Hilbert modules: let $E, F$ be two Hilbert $B$-modules over the same \mbox{\Cs algebra} $B$.
\begin{defn}
One says that an operator $T: E \rightarrow F$ is \emph{adjointable} if there exists an operator $T^*: F \rightarrow E$ such that
$$
 \inn{}{T\xi, \eta}{} = \inn{}{\xi, T^*\eta}{}, \qquad \mbox{ for all } \quad \xi \in E, \eta \in F \; . \vspace*{-10pt}
$$
\end{defn}

An adjointable operator is automatically $B$-linear and, by the Banach-Steinhaus theorem, bounded. 
The collection of adjointable operators from $E$ to $F$ is denoted $\cL_B(E,F)$.
A bounded $B$-linear operator need not be adjointable --- a simple counterexample is the inclusion $i: I \hookrightarrow B$ of a proper ideal in a unital \mbox{\Cs algebra} --- thus the need for the definition. Clearly, if $T\in \cL_B(E,F)$, then 
$T^*\in \cL_B(F,E)$. In particular, $\BndB:=\cL_B(E,E)$ is a $*$-algebra; it is in fact a \mbox{\Cs algebra} for the 
operator norm. 

There is an important class of operators which is built from `finite rank' operators.
For any $\xi \in F$ and $\eta \in E$ one defines the operator $\theta_{\xi, \eta} : E \rightarrow F$ as
\begin{equation}
\theta_{\xi, \eta} (\zeta) =  \xi  \inn{}{\eta, \zeta}{}, \qquad \mbox{ for all } \quad \zeta \in E \; . 
\end{equation}
Every such $\theta_{\xi, \eta}$ is adjointable, with adjoint $\theta_{\xi, \eta}^* := \theta_{\eta, \xi}: F \rightarrow E$. 
The closed linear subspace of $\cL_B(E,F)$ spanned by 
$\lbrace \theta_{\xi, \eta} \, \vert \, \xi, \eta \in E \rbrace$
is denoted by $\cK_B(E,F)$. In particular $\cK_B(E):= \cK_B(E,E) \subseteq \BndB$; this is a closed ideal, whose elements are referred to as \emph{compact endomorphisms.}
When possible we write $\Bnd=\BndB$ and $\K=\cK_B(E)$.

The \mbox{\Cs algebraic} dual of $E$, denoted by $E^*$, is defined as the space
\begin{equation}\label{dualE}
E^* := \lbrace \phi \in \textup{Hom}_B(E,B) \, | \, \exists \, \xi \in E \mbox{ such that } \phi(\eta) = \langle \xi,\eta \rangle \, \, \forall \eta \in E \rbrace  \, .
\end{equation}
Thus, with $\xi \in E$, if $\lambda_\xi: E \rightarrow B$ is the operator defined by $\lambda_\xi (\eta) = \inn{}{\xi, \eta}{} $, for all $\eta \in E$, 
every element of $E^*$ is of the form $\lambda_\xi$ for some $\xi\in E$.

By its definition, $E^* := \cK_B(E,B)$. One says that a module $E$ is \emph{self-dual} if the \mbox{\Cs algebraic} dual $E^*$ coincides with $\cL_B(E,B)$, i.e. if the module map given by $E \ni \xi \mapsto \lambda_\xi \in \cL_B(E,B)$,
is surjective. If $B$ is unital, then $B^n$ is self-dual for any $n\geq 1$. As a consequence, every finitely generated projective Hilbert \mbox{\Cs module} over a unital \mbox{\Cs algebra} is self-dual as well.

\subsection{Morita equivalence}
Given a right Hilbert $B$-module $E$, by construction, compact endomorphisms act from the left on $E$. 
Then, by defining:
$$
\inn{\K}{\xi,\eta}{}:= \theta_{\xi,\eta},
$$
we obtain a natural $\K$-valued Hermitian product on $E$. Note this is \emph{left} linear over $\K$, that is  $\inn{\K}{T \cdot \xi,\eta}{} = T \cdot (\inn{\K}{\xi,\eta}{})$ for $T\in \K$. 
Thus $E$ is a left Hilbert $\K$-module and by the very definition of $\K$, $E$ is full over $\K$.
One easily checks the compatibility condition
\begin{equation}\label{mekop}
\inn{\K}{\xi ,\eta }{}\zeta = \xi \inn{}{\eta, \zeta}{B}, \qquad \mbox{ for all } \quad \xi, \eta, \zeta \in E.
\end{equation}
In particular, the $B$-valued and $\K$-valued norms coincide \cite[Lem.~2.30]{RW}. 
By its definition, $\K$ acts by adjointable operators on the right \mbox{$B$-module} $E$.
On the other hand, with $b\in B$ and $\xi,\eta,\zeta\in E$, one computes:
$$
\inn{\K}{\xi b,\eta }{}\zeta = (\xi b) \inn{}{\eta, \zeta}{B} = \xi \inn{}{\eta b^*, \zeta}{B} = \inn{\K}{ \xi, \eta b^*}{} \zeta ,
$$
that is, $B$ acts by adjointable operators on the left $\K$-module $E$.

This example motivates the following definitions:
\begin{defn}
Given two \mbox{\Cs algebras} $A$ and $B$,  
a Hilbert \emph{$(A,B)$-bimodule} $E$ is a right Hilbert $B$-module with $B$-valued Hermitian product 
$\inn{}{\ ,\ }{B}$, which is at the same time a left Hilbert $A$-module with $A$-valued Hermitian product 
$\inn{A}{\ , \ }{}$ and such that the Hermitian products are compatible, that is,
\begin{equation} 
\xi \inn{}{\eta, \zeta}{B} = \inn{A}{\xi ,\eta }{}\zeta , \qquad \mbox{ for all } \quad \xi, \eta, \zeta \in E.
\end{equation}
Note that $\inn{}{\ ,\ }{B}$ is right $B$-linear, while $\inn{A}{\ , \ }{}$ is left $A$-linear. 
\end{defn}

 \begin{defn}
 An \emph{$(A,B)$-equivalence bimodule} is a Hilbert $(A,B)$-bimodule $E$ that is full both as a left and as a right Hilbert module 
 and such that 
\begin{equation}\label{coadj}
\inn{}{a \xi, \eta}{B} =  \inn{}{\xi, a^* \eta}{B} \qquad \mbox{and} \qquad \inn{A}{\xi b, \eta }{} = \inn{A}{\xi, \eta b^*}{} , 
\end{equation}
for all $\xi, \eta \in E, a \in A, b\in B$.
If there exist an $(A,B)$-equivalence bimodule one says that the two \mbox{\Cs algebras} $A$ and $B$ are \emph{Morita equivalent}.
 \end{defn}

\noindent
Condition \eqref{coadj} says that
$A$ acts by adjointable operators on $E_B$ --- that is the bimodule $E$ thought of as a right $B$-module ---, with the adjoint of $a\in A$ being its conjugated $a^*$ in the $C^*$-algebra $A$. Similarly $B$ acts by adjointable operators on ${_A}E$.  
From the considerations above on the algebra $\K$ of compact endomorphisms and in particular the compatibility condition \eqref{mekop}, it is not surprising that the algebra $\K$ has a central role for Morita equivalence:

 \begin{prop}[{\cite[Prop.~3.8]{RW}}]
Let $E$ be a full Hilbert $B$-module. Then $E$ is a $(\K,B)$-equivalence bimodule with $\K$-valued Hermitian product given by $\inn{\K}{\xi,\eta}{} = \theta_{\xi,\eta}$. Conversely, if $E$ is an $(A,B)$-equivalence bimodule, then there exists an isomorphism $\phi:A \rightarrow \K$ such that
$$
  \phi \left( \inn{A}{\xi, \eta}{} \right) = \inn{\K}{\xi, \eta}{}, \qquad \mbox{ for all } \quad \xi, \eta \in E .
 $$
  \end{prop}
Thus, two \mbox{\Cs algebras} $A$ and $B$ are Morita equivalent if and only if $A \simeq \KB$ for a full right Hilbert $B$-module $E_B$, or equivalently, if and only if $B \simeq \KA$ for a full left Hilbert $B$-module ${}_AE$. 

In fact, Morita equivalence is an equivalence relation, with transitivity implemented by the interior tensor product of bimodules. Given a Hilbert $(A,B)$-bimodule $E$ and a Hilbert $(B,C)$-bimodule $F$, one can form a Hilbert $(A,C)$-bimodule, the \emph{interior tensor product bimodule}. We do not dwell upon the details of the construction here while referring to \cite[\S3.2]{RW} for instance. We shall illustrate a particular case of the construction later on.

\subsection{Self-Morita equivalence bimodules}\label{se:smeb}
In the present paper we are interested in \emph{noncommutative line bundles}, that is the analogue of 
modules of continuous sections of a line bundle over a topological space. 
We are then naturally led to the following definition.

\begin{defn}\label{de:sme}
A \emph{self-Morita equivalence bimodule} for $B$ is a pair $(E,\phi)$ with
$E$ a full right Hilbert $B$-module $E$ and 
$\phi: B \rightarrow \K$ an isomorphism.
\end{defn}

The prototypical commutative example of a self-Morita equivalence bimodule is provided by $B=C(X)$, the \mbox{\Cs algebra} of continuous functions on a compact topological space,  $E$ the $C(X)$-module of sections of a Hermitian
line bundle $\mathcal{L} \rightarrow X$ and $\phi$ the trivial representation.

If $(E, \phi)$ is a self-Morita equivalence bimodule over $B$, the dual $E^*$ as defined in \eqref{dualE}, can be made into a self-Morita equivalence bimodule 
over $B$ as well. Firstly, $E^*$ is given the structure of a (right) Hilbert $C^*$-module over $B$ via $\phi$. 
Recall that elements of $E^*$ are of the form $\lambda_\xi$ for some $\xi\in E$, with $\lambda_\xi (\eta) = \inn{}{\xi, \eta}{} $, for all $\eta \in E$. 
The right action of $B$ on $E^*$ is given by
$$
\lambda_\xi \, b := \lambda_\xi \circ \phi(b) = \lambda_{\phi(b)^* \xi} \, ,
$$
the second equality being easily established. The $B$-valued Hermitian product on $E^*$ uses the left $\K$-valued Hermitian product on $E$:
$$
\langle \lambda_\xi, \lambda_\eta \rangle := \phi^{-1}(\theta_{\xi, \eta})  \, ,
$$
and $E^*$ is full as well. Next, define a $*$-homomorphism $\phi^* : B \to \cL(E^*)$ by 
$$
\phi^*(b)(\lambda_\xi) := \lambda_{\xi \cdot b^*} , 
$$
which is in fact an isomorphism $\phi^*: B \rightarrow \cK(E^*)$. 
Thus, the pair $(E^*,\phi^*)$ is a self-Morita equivalence bimodule over $B$, according to Definition~\ref{de:sme}. 

We need to recall the notion of interior tensor product $E\hot_\phi E$ of $E$ with itself over $B$.
%
As a first step, one considers the algebraic tensor product 
$E\otimes_{\rm alg}E$.
It has natural right $B$-module structure given by 
$$
(\xi \ots \eta) b = \xi \ots (\eta b) \, , \qquad \mbox{for}  \quad \xi, \eta \in E \, , \quad b \in B \, ,
$$
and a $B$-valued Hermitian product given, on simple tensors, by
\begin{equation}
\label{eq:interprod}
\inn{}{\xi_1 \ots \xi_2, \eta_1 \ots \eta_2}{} = \inn{}{\xi_2, \phi( \inn{}{\xi_1,\eta_1}{} ) \, \eta_2}{}
\end{equation}
and extended by linearity. This is well defined and has all required properties; in particular, the null space 
$N=\{\zeta \in E\otimes_{\rm alg} E \, ;\,  \inn{}{\zeta, \zeta}{} = 0\}$ 
is shown to coincide with the subspace generated by elements of the form
\begin{equation}\label{ns}
\xi b \ots \eta - \xi \ots \phi(b) \eta \, , \qquad \mbox{for}  \quad \xi, \eta \in E \, , \quad b \in B \, . 
\end{equation}%
One then takes $E \otimes_\phi E := (E\otimes_{\rm alg}E) / N$ and defines $E \hot_\phi E$ to be the (full) Hilbert 
\mbox{$B$-module} obtained by completing with respect to the norm induced by sthe restriction of \eqref{eq:interprod} to the quotient $E \otimes_\phi E$. 
We shall simply write $\xi \ots \eta$ to mean the element $\xi \ots \eta + N$ in $E \hot_\phi E$.

The module $E \hot_\phi E$ is itself a self-Morita equivalence bimodule over $B$. To sketch  how this is the case, we need some additional facts (\emph{cf.} \cite[Ch.4]{La95}). 

For   any $T\in\cL(E)$, the operator defined on simple tensors by $\xi \ots \eta \mapsto T(\xi) \ots \eta$ extends to a well-defined bounded operator on $E \hot_\phi E$
denoted $T \ots \id =: \phi_*(T)$. It is adjointable with adjoint given by $T^* \ots \id$. 
Next, for   $\xi \in E$, the equation $S_\xi(\eta) = \xi \ots \eta$, with $\eta\in E$, defines an element 
$S_\xi \in \cL_B(E, E \hot_\phi E)$ whose adjoint is $S_\xi^*(\eta \ots \zeta) = \phi (\hs{\xi}{\eta}) \zeta$,
for all  $\xi, \eta, \zeta \in E$. Finally, for all  $b\in B$ and $\xi,\eta,\zeta_1, \zeta_2$, one computes:
\begin{align*}
S_\xi \phi(b) S_\eta^* (\zeta_1 \ots \zeta_2) & = S_\xi \phi(b) \phi(\hs{\eta}{\zeta_1}) \zeta_2
= \xi \ots \phi(b \hs{\eta}{\zeta_1})\zeta_2 \\
& = \xi b \hs{\eta}{\zeta_1} \ots \zeta_2 = \theta_{\xi b, \eta}(\zeta_1) \ots \zeta_2 = \phi_*(\theta_{\xi b, \eta}) (\zeta_1 \ots \zeta_2) 
\end{align*}
Thus
$$
\phi_*(\theta_{\xi b, \eta}) = S_\xi \phi(b) S_\eta^* ,
$$
which is in $\cK(E \hot_\phi E)$ since $\phi(b)\in \K$. In fact, since $\phi$ is non degenerate, it follows from this (by using an approximate unit for $B$ in general) that $\phi_*(\theta_{\xi, \eta}) = S_\xi S_\eta^*$, and since $\K$ is generated by elements of the form $\theta_{\xi, \eta}$, it follows that 
$\phi_*(\K) \subset \cK(E \hot_\phi E)$. In fact, since $\phi$ is an isomorphism, $\phi_*$ is an isomorphism as well \cite[Prop.~4.7]{La95}.  In particular, generating elements of $\cK(E \hot_\phi E)$ of the form 
$\theta_{\xi_1 \ots \xi_2, \eta_1 \ots \eta_2}$ can be written as
$$
\theta_{\xi_1 \ots \xi_2, \eta_1 \ots \eta_2} = \phi_* ( \theta_{\xi_1 b, \eta_1} ),
$$
with $b\in B$ uniquely defined by $\phi(b)=\theta_{\xi_2 , \eta_2}$. 
\begin{prop}
The isomorphism
$$
\phi^{(2)} := \phi_* \circ \phi : B \to \cK(E \hot_\phi E)
$$
realizes the Hilbert $B$-module $E \hot_\phi E$ as a self-Morita equivalence over $B$. 
\end{prop}
\noindent
The construction can be iterated and, for $n>0$, one gets the $n$-fold interior tensor power of $E$ over $B$, 
$$
E^{\hot_{\phi} n}:=
 E \hot_\phi E \hot_\phi \cdots \hot_\phi E , \qquad  n\textup{-factors} ;
 $$ 
again a self-Morita equivalence bimodule over $B$.

\begin{rem}
One could generalize the previous construction and consider for a \Cs algebra $B$, the collection self-Morita equivalence bimodules over $B$ up to unitary equivalence. This has a natural group structure with respect to the interior tensor product.
The inverse of the self-Morita equivalence bimodule $(E,\phi)$ is the dual self-Morita equivalence bimodule $(E^*,\phi^*)$.
Thinking of self-Morita equivalence bimodules as line bundles, this group is 
the \emph{Picard group} of $B$, denoted $\mathrm{Pic}(B)$, in analogy with the classical  
Picard group of a space, which is the group of
isomorphism classes of line bundles with group operation given by the tensor product.
It was shown in \cite{AE97} that the Picard group of a commutative unital \Cs algebra $B=C(X)$ is the semi-direct product of the classical Picard group of $X$ with the group of automorphisms of the algebra $B$, which is the same as the group of homeomorphisms of $X$.
\end{rem}

\subsection{Frames}\label{frames}
Le $A,B$ be two unital $C^*$-algebras, and let 
$E$ be a finitely generated 
$(A,B)$-equivalence bimodule.
Since $A \simeq \K$, there exists a finite collection of elements $\eta_1, \dots, \eta_n \in E$ 
with the property that 
$$ 
\sum\nolimits_j \inn{A}{\eta_j, \eta_j}{} = 1_A.
$$
Equivalently, by using the isomorphism $\phi: A \rightarrow \K$, this means that \
\begin{equation}
\label{eq:idK}
\sum\nolimits_j \, \theta_{\eta_j, \eta_j} = 1_{\K}.
\end{equation}
As a consequence, one can reconstruct any element of $\xi \in E$ as
\begin{equation}\label{recon}
\xi = \sum\nolimits_j  \eta_j \inn{}{\eta_j, \xi}{B}.
\end{equation}
This motivates the following definition \cite{R10}.

\begin{defn}
Let $B$ be a unital \Cs algebra. A \emph{finite standard module frame} for the right Hilbert $B$-module $E$ is a finite family of elements 
$\lbrace \eta_i \rbrace_{j=1}^{n}$ of $E$ such that, for all $\xi \in E$, the reconstruction formula \eqref{recon}
is true. 
\end{defn}

\begin{rem}
More generally, one could consider frames with countable elements, with \eqref{recon} replaced by a series convergent in $E$, or equivalently \eqref{eq:idK} replaced by the condition that
the series $\sum_j\theta_{\eta_j,\eta_j}$ is strictly convergent to the unit of $\Bnd$ ($\K$ need not be unital). We refer to \cite{FL02} for details. 
\end{rem}

The existence of a finite frame is a geometrical condition. Indeed, whenever one has a right Hilbert  $B$-module $E$ with a finite standard module frame, $E$ is algebraically finitely generated and projective as a right module, with projectivity following from the fact that the algebra $B$ is unital, with the frame explicitly providing 
a projection for $E$. Indeed,
the matrix $p = (p_{jk})$ with $p_{jk} = \inn{}{\eta_j, \eta_k}{B}$ 
is a projection in the matrix algebra $\mathrm{M}_n(B)$.
By construction $(p_{jk})^* = p_{kj}$ and, using \eqref{recon},
\begin{align*}
(p^2)_{jl} & = \sum\nolimits_k \hs{\eta_j}{\eta_k}_B \hs{\eta_k}{\eta_l}_B \\  
& =\sum\nolimits_k \hs{\eta_j}{\eta_k \hs{\eta_k}{\eta_l}_B}_B 
= \hs{\eta_j}{\eta_l}_B = p_{jl} \, .
\end{align*}
This establishes the finite right $B$-module projectivity of $E$ with the isometric identification $E \simeq p B^n$. 
Furthermore, $E$ is self-dual for its Hermitian product.

More generally, $E$ is finitely generated projective whenever there exist two finite sets $\lbrace \eta_i \rbrace_{i=1}^{n}$ and $\lbrace \zeta_i \rbrace_{i=1}^{n}$ of elements of $E$ with the property that
\begin{equation} \label{eq:dualframes}
\sum\nolimits_j \inn{\K}{\eta_j, \zeta_j}{} = 1_{\K}.
\end{equation}
Then, any element $\xi \in E$ can be reconstructed as 
$$
\xi = \sum\nolimits_j \eta_j \inn{}{\zeta_j, \xi}{B},
$$
and the matrix with entries given by $\mathbf{e}_{jk}=\inn{}{\zeta_j, \eta_k}{B}$ is an idempotent in $\mathrm{M}_n(B)$, 
$(\mathbf{e}^2)_{jk}=\mathbf{e}_{jk}$, and $E\simeq \mathbf{e}B^n$ as a right $B$-module.

\section{Pimsner algebras and generalized crossed products}\label{s:2}
In this section, we recall the construction of the Pimsner algebra \cite{P97}
(a ring-theoretic version is discussed in \cite{BB14}).
We also review the notion of generalized crossed product of a \mbox{\Cs algebra} by a Hilbert bimodule, that was introduced independently in \cite{AEE}. The two notions are related and in the case of a self-Morita equivalence they actually coincide.

\subsection{The Pimsner algebra of a self-Morita equivalence}
In his breakthrough paper \cite{P97}, Pimsner associates a universal \mbox{\Cs algebra} to every pair $(E, \phi)$, with $E$ a right Hilbert $B$-module  for a \mbox{\Cs algebra} $B$ and $\phi: B \rightarrow \Bnd$ is an isometric $*$-homomorphism. His work was later generalized by Katsura \cite{Ka04}, who removed the injectivity assumption on $\phi$.
 
Guided by a geometric approach coming from principal circle bundles, we shall not work in full generality, but rather under the assumption that the pair $(E,\phi)$ is a self-Morita equivalence bimodule for $B$. Things simplify considerably and the Pimsner algebra is represented on a Hilbert module \cite{AKL14}.

Given a self-Morita equivalence bimodule $(E,\phi)$ for the \mbox{\Cs algebra} $B$, in \S\ref{se:smeb} we described the interior tensor product 
$E \hot_\phi E$, itself a self-Morita equivalence bimodule and, more generally, the tensor product module $E^{\hot_{\phi} n}$,  for $n>0$. To lighten notation, we denote
$$
E^{(n)} := \begin{cases} E^{\hot_{\phi} n} & n>0 \\
B & n=0 \\
(E^*)^{\hot_{\phi^*} (-n)} & n<0 
\end{cases} \, \, .
$$
Clearly, $E^{(1)}=E$ and $E^{(-1)}=E^*$. 
From the definition of these Hilbert $B$-modules, one has isomorphisms
$$
\cK(E^{(n)}, E^{(m)}) \simeq E^{(m-n)} \, .
$$
Out of them, one constructs the Hilbert $B$-module $\mathcal{E}_{\infty}$ as a direct sum:
$$
\mathcal{E}_{\infty}:= \bigoplus_{n \in \IZ} E^{(n)} \,,
$$ 
which could be referred to as a \emph{two-sided Fock module}. As on usual Fock spaces, one defines  
creation and annihilation operators.
Firstly, for each $\xi \in E$ one has a bounded adjointable operator (a creation operator) 
$S_\xi : \mathcal{E}_\infty \to \mathcal{E}_\infty$,
shifting the degree by $+1$, defined on simple tensors by:
\begin{align*}
S_\xi(b) & := \xi \cdot b \, , & b \in B  \, , \\
S_\xi(\xi_1 \ots \cdots \ots \xi_n) & := \xi \ots \xi_1 \ots \cdots \ots \xi_n \, , & n>0 \, , \\
S_\xi(\lambda_{\xi_1} \ots \cdots \ots \lambda_{\xi_{-n}}) & := \lambda_{\xi_2 \cdot \phi^{-1}(\theta_{\xi_1,\xi})} 
\ots \lambda_{\xi_3} \ots \cdots \ots \lambda_{\xi_{-n}}  \, , & n<0  \, .
\end{align*}
\noindent
The adjoint of $S_\xi$ (an annihilation operator) is given by $S_{\lambda_\xi} := S_\xi^* : \mathcal{E}_\infty \to \mathcal{E}_\infty$:
\begin{align*}
S_{\lambda_\xi}(b) & := \lambda_\xi \cdot b \, , & b \in B  \, ,  \\
S_{\lambda_\xi}(\xi_1 \ots \dots \ots \xi_n) & := \phi(\inn{}{\xi,\xi_1}{})\xi_2 \ots \xi_3 \ots \cdots \ots \xi_n \, , & n>0 \, , \\
S_{\lambda_\xi}(\lambda_{\xi_1} \ots \dots \ots \lambda_{\xi_{-n}}) & := \lambda_\xi \ots \lambda_{\xi_1} \ots \cdots \ots \lambda_{\xi_{-n}} \, , & n<0 \, ;
\end{align*}
In particular, $S_\xi(\lambda_{\xi_1}) = \phi^{-1}(\theta_{\xi, \xi_1}) \in B$ and $S_{\lambda_\xi}(\xi_1) = \inn{}{\xi,\xi_1}{} \in B$.

\begin{defn}
The \emph{Pimsner algebra} $\cO_E$ of the pair
$(E,\phi)$ is the smallest $C^*$-subalgebra of $\cL_B(\mathcal{E}_\infty)$ 
containing all creation operators $S_\xi$ for $\xi \in E$. 
\end{defn}

\noindent
There is an injective $*$-homomorphism $i : B \to \cO_E$. This is induced by 
the injective $*$-homomorphism $\phi : B \to \cL_B(\mathcal{E}_\infty)$ defined by
\begin{align*}
\phi(b)(b') & := b \cdot b' \, , \\ 
\phi(b)(\xi_1\otimes\cdots\otimes \xi_n) & := \phi(b)(\xi_1) \ots \xi_2 \otimes\cdots\otimes \xi_n \, , \\
\phi(b)(\lambda_{\xi_1} \otimes\cdots\otimes \lambda_{\xi_n}) & := \phi^*(b)(\lambda_{\xi_1}) \ots \lambda_{\xi_2} \otimes\cdots\otimes \lambda_{\xi_n} \\ & \:= \lambda_{\xi_1 \cdot b^*} \ots \lambda_{\xi_2} \otimes\cdots\otimes \lambda_{\xi_n} \, ,
\end{align*}
and whose image is in the Pimsner algebra $\cO_E$. 
In particular, for all $\xi,\eta \in E$ it holds that $S_\xi S_{\eta}^* = i(\phi^{-1}(\theta_{\xi,\eta}))$, that is the operator $S_\xi S_{\eta}^*$ on $\mathcal{E}_\infty$ is right-multiplication by the element $\phi^{-1}(\theta_{\xi,\eta})\in B$.

A Pimsner algebra is universal in the following sense \cite[Thm.~3.12]{P97}:

\begin{prop}\label{t:unipro}
Let $C$ be a $C^*$-algebra and $\sigma : B \to C$ a \mbox{$*$-homomorphism}. Suppose  
there exist elements $s_\xi \in C$ such that,
for all $\xi, \eta \in E$, $b \in B$ and $\alpha, \beta \in \IC$ it holds that:
\begin{align*}
& \alpha s_\xi + \beta s_\eta = s_{\alpha\xi + \beta \eta}  , \\
&  s_{\xi} \sigma (b) =s_{\xi b} \qquad \mbox{and} \qquad \sigma(b) s_{\xi} = s_{\phi(b)(\xi)} , \\
& s^*_\xi s_\eta = \sigma (\hs{\xi}{\eta}) , \\
& s_\xi s_\eta^* = \sigma\big(\phi^{-1}(\theta_{\xi,\eta}) \big)  . 
\end{align*}
Then, there exists a unique $*$-homomorphism $\wt{\sigma}: \cO_E \rightarrow C$ 
with the property that $\wt{\sigma}(S_\xi) = s_\xi$ for all $\xi \in E$.
\end{prop}

The following example was already in \cite{P97}.
\begin{ex}\label{ex:3.3}
Let $B$ be a \mbox{\Cs algebra} and $\alpha: B \rightarrow B$ an automorphism of $B$. Then $(B,\alpha)$ is naturally a self-Morita equivalence for $B$. 
The right Hilbert \mbox{$B$-module} structure is the standard one, with right $B$ valued Hermitian product $\inn{}{a,b}{B} = a^* b$.  
The automorphism $\alpha $ is used to define the left action via $a \cdot b = \alpha(a) b $ and the left $B$-valued Hermitian product by $\inn{B}{a,b}{} = \alpha(a^*b)$. 
The module $\mathcal{E}_{\infty}$ is isomorphic to a direct sum of copies of $B$. Indeed, for all $n \in \IZ$, the module $E^{(n)} $ is isomorphic to $B$ as a vector space, with 
\begin{equation}
\label{autaction}
E^{(n)} \ni a \cdot (x_1 \ots \cdots \ots x_n) \longmapsto \alpha^n(a) \alpha^{n-1}(x_1) \cdots \alpha(x_{n-1}) x_n \in B.
\end{equation}
The corresponding Pimsner algebra $\cO_E$ agrees with the crossed product algebra $B \rtimes_{\alpha} \IZ$.
\end{ex}

\begin{ex}
In the finitely generated projective case, occurring e.g. when the algebra $B$ is unital, the Pimsner algebra of a self-Morita equivalence can be realized explicitly in terms of generators and relations \cite{KPW}.
Since $E$ is finitely generated projective, it admits a finite frame $\lbrace \eta_j \rbrace_{j=1}^{n}$. 
Then, from the reconstruction formula as in \eqref{recon}, for any $b \in B$: 
$$
\phi(b) \eta_j= \sum\nolimits_{k} \eta_k \inn{}{\eta_k, \phi(b) \eta_j}{B},
$$
Then the \mbox{\Cs algebra} $\cO_E$ is the universal \mbox{\Cs algebra} generated by $B$ together with 
$n$ operators $S_1, \dots, S_n$, satisfying
\begin{align*}
S_k^* S_j & = \inn{}{\eta_k, \eta_j}{B}, \quad \sum\nolimits_j S_j S_j ^* = 1,  
\\ \quad \mbox{and} \quad 
b S_j & = \sum\nolimits_{k} S_k \inn{}{\eta_i, \phi(b) \eta_j}{B},
\end{align*}
for $b \in B$, and $j=1,\dots, n$. The generators $S_i$ are partial isometries if and only if the frame satisfies $\inn{}{\eta_k,\eta_j}{}=0$ for $k\neq j$. For $B=\IC$ and $E$ a Hilbert space of dimension $n$, one recovers the original Cuntz algebra $\cO_n$ \cite{C77}.
\end{ex}

Similarly to crossed products by the integers, Pimsner algebras can be naturally endowed with a circle action 
$\alpha: \IS^1 \rightarrow \mbox{Aut}(\cO_E)$ that turns them into $\IZ$-graded algebras. Indeed, by the universal property in Proposition~\ref{t:unipro} (with $C=\cO_E$, $\sigma=i$ 
the injection of $B$ into $\cO_E$, and $s_\xi:=z^*S_\xi$), the map 
$$
S_\xi \mapsto \alpha_z(S_\xi):=z^*S_\xi , \qquad z\in\IS^1, 
$$
extends to an automorphism of $\cO_E$. The degree $n$ part of $\cO_E$ can then be defined as usual, 
as the weight space $\{x\in\cO_E:\alpha_z(x)=z^{-n}x\}$.

\subsection{Generalized crossed products}

A somewhat better framework for understanding the relation between Pimsner algebras and algebras endowed with a circle action is that of generalized crossed products.
They were introduced in \cite{AEE} and are naturally associated with Hilbert bimodules via the notion of a covariant representation.

\begin{defn}
Let $E$ be a Hilbert  $(B,B)$-bimodule (not necessarily full). A
covariant representation of $E$ on a \mbox{\Cs algebra} $C$ is a pair $(\pi, \mathcal{T})$ where
\begin{enumerate}\itemsep=5pt
 \item $\pi : B \rightarrow C$ is a $*$-homomorphism of algebras;
 \item $\mathcal{T}:E \rightarrow C$ satisfies the conditions 
\begin{align*}
 \mathcal{T}(\xi) \pi(b) &= \mathcal{T}(\xi b) &
 \mathcal{T}(\xi)^*\mathcal{T}(\eta) &= \pi (\inn{}{\xi,\eta}{B}) \\[2pt]
 \pi(b) \mathcal{T}(\xi) &= \mathcal{T}(b \xi) &
 \mathcal{T}(\xi)\mathcal{T}(\eta)^* &= \pi (\inn{B}{\xi,\eta}{})\end{align*}
for all $b \in B$ and $\xi, \eta \in E$. \vspace*{-19pt}
\end{enumerate}
\end{defn}

\smallskip

\noindent
By \cite[Prop.~2.3]{AEE}, covariant representations always exist.  
\begin{defn}\label{de:gcp}
Let $E$ be a Hilbert $(B,B)$-bimodule.
The \emph{generalized crossed product} $B \rtimes_{E} \IZ$ of $B$ by $E$ is the universal \mbox{\Cs algebra}
generated by the covariant representations of E. 
\end{defn}

\noindent
In \cite[Prop.~ 2.9]{AEE} a generalized crossed product algebra is given as a cross sectional algebra 
(\emph{\`a la} Fell-Doran) for a suitable \mbox{\Cs algebraic} bundle over $\IZ$. 

It is worth stressing that a generalized crossed product need not be a Pimsner algebra  in general, since the representation of $B$ giving the left action need not be injective. However, by using the universal properties \ref{t:unipro}, one shows that for a self-Morita equivalence bimodule the two constructions yield the same algebra. The advantage of using generalized crossed products is that a \mbox{\Cs algebra} carrying a circle action that satisfies a suitable completeness condition, can be \emph{reconstructed} as a generalized crossed product.

\subsection{Algebras and circle actions}\label{s:acircle}
Let $A$ be a $C^*$-algebra and $\{\sigma_z\}_{z \in \IS^1}$ be a strongly continuous action of the circle $\IS^1$ 
on $A$. For each $n \in \IZ$, one defines the spectral subspaces
$$
A_{n} := \big\{ \xi \in A \mid \sigma_z(\xi) = z^{-n} \, \xi \ \ \mbox{ for all } z \in \IS^1 \big\} \, .
$$
Clearly, the invariant subspace $A_{0} \subseteq A$ is a \mbox{\Cs subalgebra} of $A$, with unit whenever $A$ is unital; this is the \emph{fixed-point subalgebra}. Moreover, the subspace $A_{n}A_{m}$, meant as the \emph{closed} linear span of the set of products $xy$ with $x \in A_{n}$ and $y \in A_{m}$, 
is contained in $A_{n+m}$. Thus, the algebra $A$ is $\IZ$-graded and the grading is compatible with the involution, 
that is $A_{n}^{*} = A_{-n}$ for all $n\in\IZ$.

In particular, for any $n \in \mathbb{Z} $ the space $A_{n}^*A_{n}$ is a closed two-sided ideal in $A_{0}$. 
Thus, each spectral subspace  $A_{n}$ has a natural structure of Hilbert $A_{0}$-bimodule (not necessarily full) with left and right Hermitian products:
\begin{equation}\label{bim0}
\inn{A_{0}}{x,y}{} = xy^* , \qquad \inn{}{x,y}{A_{0}} =x^*y , \quad \mbox{ for all } x, y \in A_{n}.
\end{equation}

The description via spectral subspaces allows one to characterize algebras that are obtained as generalized crossed products in the sense of Definition~\ref{de:gcp}, in terms of a quite manageable necessary and sufficient condition.

\begin{thm}[{\cite[Thm.~3.1]{AEE}}] \label{thm:ss}
Let $A$ be a \mbox{\Cs algebra} with a strongly continuous action of the circle. The algebra $A$ is isomorphic to $A_{0}\rtimes_{A_{1}} \IZ$ if and only if $A$ is generated, as a \mbox{\Cs algebra}, by the fixed point algebra $A_{0}$ and the first spectral subspace $A_{1}$ of the circle action. 
\end{thm}

The above condition was introduced in \cite{E94} and is referred to as having a \emph{semi-saturated} action. It is fulfilled in a large class of examples, like crossed product by the integers, and noncommutative (or quantum) principal circle bundles, as we shall see quite explicitly in \S\ref{se:3} below. In fact, this condition encompasses more general non-principal actions, which are however beyond the scope of the present paper.

In Theorem~\ref{thm:ss} a crucial role is played by the module $A_{1}$. If we assume that it is a full bimodule, that is if 
\begin{equation}\label{eq:lss}
A_{1}^*A_{1} =A_{0} = A_{1}A_{1}^*, 
\end{equation}
the action $\sigma$ is said to have \emph{large spectral subspaces} (\textit{cf.} \cite[\S 2]{P81}), a slightly stronger condition than semi-saturatedness (\emph{cf.} {\cite[Prop.~ 3.4]{AKL14}}). Firstly, the condition above is equivalent to the condition that each bimodule $A_{n}$ is full, that is \[A_{n}^*A_{n} =A_{0} = A_{n}A_{n}^*\] for all $n\in\IZ$.
When this happens, all bimodules $A_{n}$ are self-Morita equivalence bimodules
for $A_{0}$, with isomorphism $\phi : A_{0} \to \cK_{A_{0}}(A_{n}) $ given by 
\begin{equation}
\label{phi}
\phi(a)(\xi) := a \, \xi, \qquad \mbox{ for all } a \in A_{0}, \,  \xi \in A_{n}.
\end{equation}

Combining Theorem~\ref{thm:ss} with the fact that for a self-Morita equivalence the generalized crossed product construction and Pimsner's construction yield the same algebra, we obtain the following result.

\begin{thm}\label{t:pimcir}{\cite[Thm.~ 3.5]{AKL14}}
Let $A$ be a \mbox{\Cs algebra} with a strongly continuous action of the circle. Suppose that the first spectral subspace $A_{1}$ is a full and countably generated Hilbert bimodule over $A_{0}$.
Then the Pimsner algebra $\cO_{A_{1}}$  of the self-Morita equivalence $(A_{1}, \phi)$, with $\phi$ as in is \eqref{phi}, is isomorphic to $A$. The isomorphism is given by $S_\xi \mapsto \xi$ for all $\xi \in A_{1}$.
\end{thm}

Upon completions, all examples considered in the present paper will fit into the framework of the previous theorem. 

\subsection{Six term exact sequences}\label{se:stes}
With a Pimsner algebra there come two natural six term exact sequences in $KK$-theory, which relate the $KK$-groups of the Pimsner algebra $\cO_E$ with those of the $C^*$-algebra of (the base space) scalars $B$. The corresponding sequences in $K$-theory are noncommutative analogues of the Gysin sequence which in the commutative case relates the $K$-theories of the total space and of the base space of a principal circle bundle. The classical cup product with the Euler class is replaced, in the noncommutative setting, by a Kasparov product with the identity minus the generating Hilbert $C^*$-module $E$.

Firstly, since $\phi : B \to \cK(E) \subseteq \cL(E)$, the following class is well defined.

\begin{defn}
The class in $KK_0(B,B)$ defined by the even Kasparov module $(E, \phi, 0)$ (with trivial grading) will be denoted by $[E]$.
\end{defn}

\noindent
 Next, consider the orthogonal projection $P : \mathcal{E}_\infty \to \mathcal{E}_\infty$ with range 
\[
\mbox{Im}(P) =  \bigoplus_{n \geq 0 }^\infty E^{(n)} \subseteq \mathcal{E}_\infty \, .
\]
Since $[P,S_\xi] \in \cK(\mathcal{E}_\infty)$ for all $\xi \in E$, one has $[P,S] \in \cK(\mathcal{E}_\infty)$ for all $S \in \cO_E$. 
Then, let $F := 2P-1 \in \cL(\mathcal{E}_\infty)$ and let $\wt \phi : \cO_E \to \cL(\mathcal{E}_\infty)$ be the inclusion.

\begin{defn}
The class in $KK_1(\cO_E,B)$ defined by the odd Kasparov module $(\mathcal{E}_\infty, \wt \phi, F)$ will be denoted by $[\partial]$.
\end{defn}

\noindent
For any separable $C^*$-algebra $C$ we then have the group homomorphisms
\[
[E] : KK_*(B,C) \to KK_*(B,C)\, , \quad [E] : KK_*(C,B) \to KK_*(C,B) \\
\]
and
\[
[\partial] : KK_*(C,\cO_E) \to KK_{* + 1}(C,B) \, , \quad [\partial] : KK_*(B,C) \to KK_{* + 1}(\cO_E,C) \, ,
\]
which are induced by the Kasparov product.

These yield natural six term exact sequences in $KK$-theory \cite[Thm.~4.8]{P97}. We
report here the case $C = \IC$. Firstly, the sequence in $K$-theory: 
\[
\begin{CD}
K_0(B) @>{1 - [E]}>> K_0(B) @>{i_*}>> K_0(\cO_E) \\
@A{[\partial]}AA & & @VV{[\partial]}V \quad . \\
K_1(\cO_E) @<<{i_*}< K_1(B) @<<{1 - [E]}< K_1(B) 
\end{CD}
\]
with $i_*$ the homomorphism in $K$-theory induced by the inclusion $i: B \to \cO_E$.
This could be considered as a generalization of the classical \emph{Gysin sequence} in $K$-theory (see \cite[IV.1.13]{Ka78}) 
for the `line bundle' $E$ over the `noncommutative space' $B$ and with the map $1 - [E]$ having the same role as the 
\emph{Euler class} $\chi(E):=1 - [E]$ of the line bundle $E$. 

The second sequence would then be an analogue in $K$-homology:
\[
\begin{CD}
K^0(B) @<<{1 - [E]}< K^0(B) @<<{i^*}< K^0(\cO_E) \\
@VV{[\partial]}V & & @A{[\partial]}AA \quad . \\
K^1(\cO_E) @>{i^*}>> K^1(B) @>{1 - [E]}>> K^1(B)
\end{CD}
\]
where now $i^*$ is the induced homomorphism in $K$-homology.

Gysin sequences in $K$-theory were given in \cite{ABL14} for line bundles over quantum projective spaces and leading to a class of quantum lens spaces. 
These examples were generalized in \cite{AKL14} to a class of quantum lens spaces as circle bundles over quantum weighted projective lines with arbitrary weights.
 
\section{Principal bundles and graded algebras}\label{se:3}

Examples of Pimsner algebras come from noncommutative (or quantum) principal circle bundles. At an algebraic level the latter are intimately related to  $\IZ$-graded $*$-algebras. When completing with natural norms one is lead to continuous circle actions on a \mbox{\Cs algebra}  with $\IZ$-grading given by spectral subspaces, that is the framework described in \S\ref{s:acircle}.

\subsection{Noncommutative principal circle bundles and line bundles}
We aim at exploring the connections between (noncommutative) principal circle bundles, frames for modules as described in \S\ref{frames}, and $\IZ$-graded algebras. The circle action is dualized in a coaction of the dual group Hopf algebra. 
Thus, we need to consider the unital complex algebra 
$$
\cO(U(1)) := \IC[z,z^{-1}]/\inn{}{1 - z z^{-1}}{} ,
$$ 
where $\inn{}{1 - z z^{-1}}{}$ is the ideal generated by $1 - z z^{-1}$ in the polynomial algebra $\IC[z,z^{-1}]$ on two variables. The algebra $\cO(U(1))$ is a Hopf algebra by defining, for any $n \in \IZ$, the coproduct $\Delta : z^n \mapsto z^n \ots z^n$, the antipode $S : z^n \mapsto z^{-n}$ and the counit $\epsilon : z^n \mapsto 1$. 

Let $\mathcal{A}$ be a complex unital algebra and suppose in addition it is a right comodule algebra over $\cO(U(1))$, that is 
$\mathcal{A}$ carries a coaction of $\cO(U(1))$,
$$
\Delta_R : \mathcal{A} \to \mathcal{A} \ots \cO(U(1))  \, ,
$$
a homomorphism of unital algebras. Let
$\mathcal{B} := \{ x \in \mathcal{A} \mid \Delta_R(x) = x \ots 1 \}$ be the unital subalgebra of $\mathcal{A}$ made of coinvariant elements for $\Delta_R$.

\begin{defn}\label{de:can}
One says that the datum $\big( \mathcal{A}, \cO(U(1)), \mathcal{B} \big)$ is a \emph{noncommutative (or quantum) principal $U(1)$-bundle} when the \emph{canonical map}
$$
\mathrm{can} : \mathcal{A} \otimes_{\mathcal{B}} \mathcal{A} \to \mathcal{A} \ots \cO(U(1))  \, , \quad x \ots y \mapsto x \, \Delta_R(y) \, ,
$$
is an isomorphism.
\end{defn}

\noindent
In fact, the definition above is the statement that the right comodule algebra $\A$ is a $\cO(U(1))$ Hopf-Galois extension of $\mathcal{B}$, and this is equivalent (in the present context) by \cite[Prop.~1.6]{Haj:SCQ} to the bundle being a noncommutative principal bundle for the universal differential calculus in the sense of \cite{BM93}.

Next, let $\mathcal{A} = \oplus_{n \in \IZ} \mathcal{A}_{n}$ be a $\IZ$-graded unital algebra. The unital algebra homomorphism,
$$
\Delta_R : \mathcal{A} \to \mathcal{A} \ots \cO(U(1)) , \quad x \mapsto x \ots z^{-n} \, , \, \, \, \mbox{for} \, \, \, x \in 
\mathcal{A}_{n} \, 
$$
turns $\mathcal{A}$ into a right comodule algebra over $\cO(U(1))$. Clearly the unital subalgebra of coinvariant elements coincides with $\mathcal{A}_{0}$.
 
We present here a necessary and sufficient condition for the corresponding canonical map as in 
Definition~\ref{de:can} to be bijective \cite[Thm.~4.3]{AKL14} (\emph{cf.}  also \cite[Lem.~5.1]{SV15}). 
This condition is more manageable in general, and in particular it can be usefully applied for examples like the quantum lens spaces as principal circle bundles over quantum weighted projective lines \cite{AKL14,DAL14}. 

\begin{thm}\label{quapriseq}
The triple $\big( \mathcal{A}, \cO(U(1)), \mathcal{A}_{0} \big)$ is a noncommutative principal $U(1)$-bundle if and only if there exist finite sequences
$$
\{ \xi_j \}_{j = 1}^N \, , \, \, \{\beta_i\}_{i = 1}^M \mbox{ in } \mathcal{A}_{1} \quad \mbox{and} \quad
\{\eta_j\}_{j = 1}^N \, , \, \, \{\alpha_i\}_{i = 1}^M \mbox{ in } \mathcal{A}_{-1}
$$
such that one has identities:
\begin{equation}\label{fr1u}
\sum\nolimits_{j = 1}^N \eta_j \xi_j = 1_{\mathcal{A}} = \sum\nolimits_{i = 1}^M \alpha_i \beta_i \, .
\end{equation}
\end{thm}
\noindent
Out of the proof in \cite[Thm.~4.3]{AKL14} we just report the explicit form of the inverse map 
$\mathrm{can}^{-1} : \A \ots \cO(U(1)) \to \A \ots_{\A_{0}} \A$, given by the formula
\begin{equation} \label{caninvn}
\mathrm{can}^{-1} : x \ots z^n \mapsto 
\begin{cases}
\sum_{j_{k=1}}^N \, x \, \xi_{j_1} \cdot\ldots\cdot \xi_{j_n} \ots \eta_{j_n} \cdot\ldots\cdot \eta_{j_1} \, ,
& \,\, n < 0  \\[2pt]
x\otimes 1 \,,& \,\, n = 0\\[2pt]
\sum_{i_{k=1}}^{M} \, x \, \alpha_{i_1} \cdot\ldots\cdot \alpha_{i_{-n}} \ots \beta_{i_{-n}} \cdot\ldots\cdot \beta_{i_1} \, ,
& \,\, n > 0
\end{cases} \, \,.
\end{equation}
Now, \eqref{fr1u} are exactly the frame relations \eqref{eq:dualframes} for $\A_1$ and $\A_{-1}$, which imply that they are finitely generated and projective over $\A_{0}$ \cite[Cor.~4.5]{AKL14}.

Explicitly, with the $\xi$'s and the $\eta$'s as above, one defines the module homomorphisms
\begin{align*}
& \Phi_{(1)}   : \A_{1}  \to (\A_{0})^N \, ,  \\
& \Phi_{(1)}(\zeta)   = ( \eta_1 \zeta \, ,  \eta_2 \, \zeta \, , \cdots \, , \eta_N \, \zeta )^{tr} \\
\intertext{and}
& \Psi_{(1)}  : (\A_{0})^N \to \A_{1}  \, , \\ 
& \Psi_{(1)} (x_1 \, , x_2  \, , \cdots \, , x_N )^{tr}    = \xi_1 \, x_1 + \xi_2 \, x_2 + \cdots + \xi_N \, x_N  \, .
\end{align*}
It then follows that $\Psi_{(1)} \Phi_{(1)} = \mbox{id}_{\A_{1} }$. 
Thus $\mathbf{e}_{(1)} := \Phi_{(1)}\Psi_{(1)}$ is an idempotent in $M_N(\A_{0})$, 
and $\A_{1} \simeq \mathbf{e}_{(1)} (\A_{0})^N$. Similarly, 
with the $\alpha$'s and the $\beta$'s as above, one defines the module homomorphisms
\begin{align*}
& \Phi_{(-1)}   : \A_{1}  \to (\A_{0})^M \, ,  \\
& \Phi_{(-1)}(\zeta)   = ( \beta_1 \zeta \, ,  \beta_2 \, \zeta \, , \cdots \, , \beta_M \, \zeta )^{tr} \\
\intertext{and}
& \Psi_{(-1)}  : (\A_{0})^M \to \A_{1}  \, , \\ 
& \Psi_{(-1)} (x_1 \, , x_2  \, , \cdots \, , x_M )^{tr}    = \alpha_1 \, x_1 + \alpha_2 \, x_2 + \cdots + \alpha_M \, x_M  \, .
\end{align*}
Now one checks that $\Psi_{(-1)}\Phi_{(-1)}= \mbox{id}_{\A_{-1} }$. 
Thus $\mathbf{e}_{(-1)}:= \Phi_{(-1)}\Psi_{(-1)}$ is an idempotent in $M_M(\A_{0})$, and
$\A_{-1} \simeq \mathbf{e}_{(-1)} (\A_{0})^M$.

\subsection{Line bundles}
In the context of the previous section, the modules $\A_{1}$ and $\A_{-1}$ emerge as \emph{line bundles} 
over the noncommutative space dual to the algebra $\A_{0}$. In the same vein all modules $\A_{n}$ for 
$n\in\IZ$ are line bundles as well. 

Given any natural number $d$ consider the $\IZ$-graded unital algebra 
\begin{equation}\label{eq:Zd}
\mathcal{A}^{\IZ_d} := \oplus_{n \in \IZ} \mathcal{A}_{dn} ,
\end{equation} 
which can be seen as a fixed point algebra for an action of $\IZ_d := \IZ/ d\IZ$ 
on the starting algebra $\mathcal{A}$. As a corollary of Theorem~\ref{quapriseq} one gets the following:

\begin{prop}\label{p:quaprifin}
Suppose that $\big( \mathcal{A}, \cO(U(1)), \mathcal{A}_{0} \big)$ is a noncommutative principal $U(1)$-bundle. 
Then, for all $d\in \IN$, the datum $\big( \mathcal{A}^{\IZ_d}, \cO(U(1)), \mathcal{A}_{0} \big)$ is a noncommutative principal $U(1)$-bundle as well. 
\end{prop}

The proof of this result goes along the line of Theorem~\ref{quapriseq} and shows also that 
the right modules $\A_{d}$ and $\A_{-d}$ are finitely generated projective over $\A_{0}$ for all $d \in \IN$.
Indeed, let the finite sequences $\{ \xi_j\}_{j = 1}^N$, $\{\beta_i\}_{i = 1}^M$ in $\A_{1}$ and $\{\eta_j\}_{j = 1}^N$, 
$\{\alpha_i\}_{i = 1}^M$ in $\A_{-1}$ be as in Theorem~\ref{quapriseq}. Then,   
for each multi-index $J \in \{1,\ldots,N\}^d$ and each multi-index $I \in \{1,\ldots,M\}^d$  the elements
\begin{align*}
& \xi_J := \xi_{j_1} \cdot\ldots\cdot \xi_{j_d} \, , \quad  \beta_I := \beta_{i_d} \cdot\ldots\cdot \beta_{i_1} \in \A_{d} \quad \mbox{and} \\
& \eta_J := \eta_{j_d} \cdot\ldots\cdot \eta_{j_1} \, , \quad \alpha_I := \alpha_{i_1} 
\cdot\ldots\cdot \alpha_{i_d} \in \A_{-d} \, ,
\end{align*}
are clearly such that
$$
\sum_{J \in \{1,\ldots,N\}^d } \xi_J \, \eta_J =
1_{\mathcal{A}^{\IZ_d}} =
\sum_{I \in \{1,\ldots,M\}^d } \alpha_I \, \beta_I \ .
$$
These allow one on one hand to apply Theorem~\ref{quapriseq} to show principality and on the other hand to construct idempotents $\mathbf{e}_{(-d)}$ and $\mathbf{e}_{(d)}$, thus showing the finite projectivity of the right modules $\A_{d}$ and $\A_{-d}$ for all $d \in \IN$.

\subsection{Strongly graded algebras}

The relevance of graded algebras for noncommutative principal bundles was already shown in \cite{U81}. If $G$ is any (multiplicative) group with unit $e$, an algebra $\A$ is a \emph{$G$-graded algebra} if its admits a direct sum decomposition labelled by elements of $G$, that is $\A =\oplus_{g\in G} \A_g$, with the property that  $\A_g \A_h \subseteq \A_{gh}$, for all $g,h\in G$. If $\cH:= \IC G$ denotes the group algebra, it is well know that $\A$ is $G$-graded if and only if $\A$ is a right $\cH$-comodule algebra
for the coaction $\delta : \A \to \A \ots \cH$ defined on homogeneous elements $a_g \in \A_g$ by $\delta(a_g) = a_g \ots g$. Clearly, the coinvariants are given by $\A^{co\cH} = \A_e$, the identity components. One has then the following result 
(\emph{cf.} \cite[8.1.7]{M93}):
\begin{thm}\label{gradpr}
 The datum $\big(\A, \cH, \A_{e} \big)$ is a noncommutative principal $\cH$-bundle for the canonical map
 $$
 \mathrm{can} : \A \ots_{\A_{e}} \A \to \A \ots \cH  \, , \quad a \ots b \mapsto \sum\nolimits_{g} a b_g \ots g \, ,
 $$
 if and only if $\A$ is \emph{strongly graded}, that is $\A_g \A_h = \A_{gh}$, for all $g,h\in G$.
\end{thm}
For the proof, one first notes that $\A$ being strongly graded is equivalent to $\A_g \A_{g^{-1}}=\A_{e}$, for all $g \in G$.
Then one proceeds in constructing an inverse of the canonical map as in \eqref{caninvn}. 
Since, for each $g\in G$, the unit $1_{\A} \in \A_e = \A_{g^{-1}} \A_g$, there exists  
$ \xi_{g^{-1},j} $ in $\A_{g}$ and $\eta_{g,j} \in \A_{g^{-1}}$, such that 
$\sum\nolimits_j  \eta_{g,j} \xi_{g^{-1},j} = 1_\A$. 
Then, $\mathrm{can}^{-1} : \A \ots \cH \to \A \ots_{\A_e} \A$, is given by 
$$
\mathrm{can}^{-1} : a \ots g \mapsto  \sum\nolimits_j a \, \xi_{g^{-1},j} \ots \eta_{g,j} . 
$$
In the particular case of $G=\mathbb{Z}= \wh{U(1)}$, so that $\IC G=\cO(U(1))$, Theorem~\ref{gradpr} 
translates Theorem~\ref{quapriseq} into the following:

\begin{cor}
The datum $\big( \mathcal{A}, \cO(U(1)), \mathcal{A}_{0} \big)$ is a noncommutative principal $U(1)$-bundle if and only if the 
algebra $\mathcal{A}$ is strongly graded over $\IZ$, that is $\mathcal{A}_n\mathcal{A}_m=\mathcal{A}_{n+m}$,  for 
all $n, m \in \IZ$.
\end{cor}

In the context of strongly graded algebras, the fact that all right modules $\A_{n}$ for all $n \in \IZ$ are finite projective is a consequence of \cite[Cor.~I.3.3]{NaVO82}.  
 
\subsection{Pimsner algebras from principal circle bundles}

From the considerations above --- and in particular, if one compares \eqref{eq:lss} and Theorem~\ref{quapriseq} ---, it is clear that  a \mbox{\Cs algebra} $A$ is strongly $\IZ$-graded if and only if it carries a circle action with large spectral subspaces. One is then naturally led to consider Pimsner algebras coming from principal circle bundles. The context of Pimsner algebras allows for the use of the six term exact sequences in $KK$-theory,  as described in \S\ref{se:stes}, which relate the $KK$-theories of the total space algebra to that of the base space algebra. 

For commutative algebras this connection was already in \cite[Prop.~ 5.8]{GG13} with the following result: 

\begin{prop}
Let $A$ be a unital, commutative \mbox{\Cs algebra} carrying a circle action. Suppose that the first spectral subspace $E=A_{1}$ is finitely generated projective over $B=A_{0}$. Suppose furthermore that $E$ generates $A$ as a \mbox{\Cs algebra}. Then the following facts hold:
\begin{enumerate}\itemsep=3pt
\item $B=C(X)$ for some compact space $X$;
\item $E$ is the module of sections of some line bundle $L \to X$;
\item $A=C(P)$, where $P \rightarrow X$ is the principal $\IS^1$-bundle over $X$ 
associated with the line bundle $L$, and the circle action on $A$ comes from the principal $\IS^1$-action on $P$.
\end{enumerate}
\end{prop}

More generally, let us start with $\mathcal{A} = \oplus_{n \in \IZ} \mathcal{A}_n$ a graded $*$-algebra. 
Denote by $\sigma$ the circle action coming from the grading. In addition, 
suppose there is a \mbox{\Cs norm} on $\A$, and that $\sigma$ is isometric with respect to this norm: 
\begin{equation}
\label{eq:isomaction}
 \| \sigma_{z}(a) \| = \| a \|, \qquad \mbox{ for all } z \in \IS^1, \ a \in \A.
 \end{equation}
Denoting by $A$ the completion of $\mathcal{A}$,  one has the following \cite[\S3.6]{AKL14}:

\begin{lem}\label{l:denspe} 
The action $\{ \sigma_z\}_{z \in \IS^1}$ extends by continuity to a strongly continuous action of $\IS^1$ on $A$. 
Furthermore, each spectral subspace $A_{n}$ for the extended action agrees with the closure of $\mathcal{A}_{n} \subseteq A$. 
\end{lem}

The left and right Hermitian product as in \eqref{bim0} 
will make each spectral subspace $A_{n}$ a (not necessarily full) Hilbert \mbox{\Cs module} over $A_{0}$. These become full exactly 
when $\A$ is strongly graded. Theorem~\ref{t:pimcir} leads then to:

\begin{prop}\label{t:pimcirsub}
Let $\mathcal{A} = \oplus_{n \in \IZ} \mathcal{A}_{n}$ be a strongly graded $*$-algebra satisfying the  assumptions of Lemma~\ref{l:denspe}. 
Then its $C^*$-closure $A$ is generated, as a \mbox{\Cs algebra}, by $A_1$, and $A$ is isomorphic to the Pimsner algebra 
$\cO_{A_{1}}$ over $A_0$. \end{prop}

\section{Examples}\label{se:4}
As illustrated by Proposition~\ref{t:pimcirsub},
\mbox{\Cs algebras} coming from noncommutative principal circle bundles provide a natural class of examples of Pimsner algebras. In this section, we describe in details some classes of examples.

\subsection{Quantum weighted projective and quantum lens spaces}
Let $0<q<1$.
We recall from \cite{VS91} that the coordinate algebra of the unit quantum sphere $\IS^{2n+1}_q$ is the $*$-algebra $\A(\IS^{2n+1}_q)$ generated by 
$2n+2$ elements 
$\{z_i,z_i^*\}_{i=0,\ldots,n}$ subject to the relations:
\begin{align*} 
z_iz_j &=q^{-1}z_jz_i && 0\leq i<j\leq n \;,  \\
z_i^*z_j &=qz_jz_i^*  &&  i\neq j \;,  \\
[z_n^*,z_n] =0, \qquad [z_i^*,z_i] &=(1-q^2)\sum_{j=i+1}^n z_jz_j^* && i=0,\ldots,n-1 \;,   
\end{align*}
and a sphere relation:
$$
z_0z_0^*+z_1z_1^* +\ldots+z_nz_n^*=1 \;.
$$
The notation of \cite{VS91} is obtained by setting $q=e^{h/2}$.

A \emph{weight vector} $\mv{\ell}=(\ell_0,\dots,\ell_n)$ is a finite sequence of positive integers, called \emph{weights}. A weight vector is said to be \emph{coprime} if $\gcd(\ell_0,\dots,\ell_n)=1$; and it is \emph{pairwise coprime} if $\gcd(\ell_i,\ell_j)=1$, for all $i\neq j$.

For any coprime weight vector $\mv{\ell}=(\ell_0,\dots,\ell_n)$, one defines a \mbox{$\IZ$-grading} on the coordinate algebra $\mathcal{A}(\IS^{2n+1}_q)$ by declaring each $z_i$ to be of degree $\ell_i$ and $z_i^*$ of degree $-\ell_i$. 
This grading is equivalent to the one associated with the weighted circle action on the quantum sphere given by
\begin{equation}
\label{wu1act}
(z_0,z_1, \dots, z_n) \mapsto (\lambda^{\ell_0} z_0, \lambda^{\ell_1} z_1, \dots, \lambda^{\ell_n} z_n ), \qquad \lambda \in \IS^1.
\end{equation}

The degree zero part or, equivalently, the fixed point algebra for the action, is the coordinate algebra of the \emph{quantum $n$-dimensional weighted projective space} associated with the weight vector $\mv{\ell}$, and denoted by $\mathcal{A}(\WP_q^n(\mv{\ell}))$. 

\begin{rem}
For $\mv{\ell}=(1,\ldots,1)$ one gets 
the coordinate algebra $\A(\mathbb{CP}^n_q)$ of the ``un-weighted'' quantum projective space $\CP^n_q$. This is the  $*$-subalgebra of $\A(\IS^{2n+1}_q)$ generated by the elements $p_{ij}:=z_i^*z_j$ for $i,j=0,1,\ldots,n$.
\end{rem}

For a fixed positive integer $N$, one defines the coordinate algebra of the quantum lens space $\mathcal{A}(L^{2n+1}_q(N;\mv{\ell}))$ as
\begin{equation}
\mathcal{A}(L^{2n+1}_q(N; \mv{\ell})) := \mathcal{A}(\IS^{2n+1}_q)^{\IZ_N} 
\end{equation}
with the same decomposition and notation as in \eqref{eq:Zd}. 
Equivalently, $\mathcal{A}(L^{2n+1}_{q} (N; \mv{\ell}))$ is the invariant subalgebra of $\mathcal{A}(\IS^{2n+1}_q)$ with respect to the cyclic action obtained by restricting \eqref{wu1act} to the subgroup $\IZ_{N} \subseteq \IS^1$.

\begin{prop}[\textnormal{\emph{cf.} \cite{AKL14,DAL14,BF14}}]
For all weight vectors $\mv{\ell}$, let $N_{\mv{\ell}}:=\prod_{i}\ell_i$. Then the triple $ \big( \mathcal{A}(L^{2n+1}_{q} (N_{\mv{\ell}}; \mv{\ell})), \cO(U(1)), \mathcal{A}(\WP^n_q(\mv{\ell})) \big)$ is a noncommutative principal $U(1)$-bundle.
\end{prop}

\begin{ex}
A class of examples of the above construction is that of multidimensional teardrops \cite{BF14}, that are obtained for the weight vector $\mv{\ell}=(1,\dots, 1, m)$ having all but the last entry equal to 1. 
\end{ex}

Fix an integer $d \geq 1$. From Proposition~\ref{p:quaprifin} applied to the algebra
$\A(L^{2n+1}_q(dN_{\mv{\ell}}; \mv{\ell}))=\A(L^{2n+1}_q(N_{\mv{\ell}}; \mv{\ell}))^{\IZ_{d}}$ 
we get:

\begin{prop}\label{prop:5.4}
The datum $ \big(\, \mathcal{A}(L^{2n+1}_{q} (d N_{\mv{\ell}}; \mv{\ell} )) \,,\, \cO(U(1)) \,,\, \mathcal{A}(\WP^n_q(\mv{\ell})) \,\big)$ is a noncommutative principal $U(1)$-bundle for all integers $d \geq 1$.
\end{prop}

Let $C(\IS^{2n+1}_q)$ be the $C^*$-completion of $\A(\IS^{2n+1}_q)$ in the universal $C^*$-norm. By universality, one extends the $\IS^1$ weighted action to the $C^*$-algebra and defines $C(\WP^n_q(\mv{\ell}))$ and $C(L^{2n+1}_{q}(N;\mv{\ell}))$ as the fixed point $C^*$-subalgebras for this action of $\IS^1$ and of the subgroup $\mathbb{Z}_N\subset\IS^1$, respectively.
At least for $n=1$ the fixed point $C^*$-subalgebras coincide with the corresponding 
universal enveloping $C^*$-algebras (\emph{cf.} Lemma 5.6 and Lemma 6.7 of \cite{AKL14}).

Let
$E$ be the first spectral subspace of $C(\IS^{2n+1}_q)$ for the weighted action of $\IS^1$. From Prop.~\ref{prop:5.4} we then get:

\begin{prop} For any $d\geq 1$, the $C^*$-algebra
$C(L^{2n+1}_{q} (dN_{\mv{\ell}}; \mv{\ell}))$ is the Pimsner algebra over $C(\WP^n_q(\mv{\ell}))$ associated to the module $E$ above.
\end{prop}

\noindent
As particular cases, $C(\IS^{2n+1}_q)$ is a Pimsner algebra over $C(\CP^n_q)$, and more generally $C(L^{2n+1}_q(d; \mv{1}))$ is a Pimsner algebra over $C(\CP^n_q)$ for any $d\geq 1$.

\subsection{Twisting of graded algebras}
A second class of examples comes from twisting the product of the algebras of a given principal bundle by an automorphism.
Then let $\mathcal{A} = \oplus_{n \in \IZ} \mathcal{A}_n$ be a $\IZ$-graded unital $*$-algebra.

\begin{defn}
Let $\gamma$ be a graded unital $*$-automorphism of $\A$.
A new unital graded $*$-algebra $(\A,\star_\gamma)=:\B=\bigoplus_{n\in\IZ}\B_n$ is defined as follows: $\B_n=\A_n$
as a vector space, the involution is unchanged, and the product defined by:
\begin{equation}\label{tp}
a\star_\gamma b=\gamma^n(a)\gamma^{-k}(b) \;,\qquad\mbox{ for all } \;a\in\B_k,\,b\in\B_n,
\end{equation}
where the product on the right hand side is the one in $\A$.
\end{defn}

\noindent
It is indeed straightforward to check that the new product satisfies
\begin{itemize}\itemsep=0pt
\item[i)] associativity: for all $a\in\A_k$, $b\in\A_m$, $c\in\A_n$ it holds that 
$$
(a\star_\gamma b)\star_\gamma c=a\star_\gamma (b\star_\gamma c)
=\gamma^{m+n}(a)\gamma^{n-k}(b)\gamma^{-k-m}(c) ,
$$
\item[ii)] $(a\star_\gamma b)^*=b^*\star_\gamma a^*$, for all $a,b$. 
 \end{itemize}
Furthermore, the unit is preserved, that is: $1\star_\gamma a=a\star_\gamma 1=a$ for all $a$ and the degree zero
subalgebra as undeformed product: $\B_0=\A_0$. Finally,  
$$
a\star_\gamma \xi=\gamma^n(a)\xi \;, \qquad \xi\star_\gamma a=\xi\gamma^{-n}(a) \; , \qquad
\mbox{ for all }  a\in\B_0,\,\xi\in\B_n . 
$$
Thus the left $\B_0$-module structure of $\B_n$ is the one of $\A_n$ twisted with $\gamma^n$,
and the right $\B_0$-module structure is the one of $\A_n$ twisted with $\gamma^{-n}$. 

We write this as
$\B_n={_{\gamma^n}}(\A_n){_{\gamma^{-n}}}$. 

\begin{rem}
For the particular case when $\A$ is commutative, from the deformed product \eqref{tp} one gets commutation rules:
\begin{equation}\label{eq:comm}
a\star_\gamma b=\gamma^{-2k}(b)\star_\gamma\gamma^{2n}(a) \;, 
\end{equation}
for all $a\in\B_k,\,b\in\B_n$.
\end{rem}

\begin{thm}
Assume the datum $\big( \mathcal{A}, \cO(U(1)), \mathcal{A}_{0} \big)$ is a noncommutative principal $U(1)$-bundle. Then, the datum $\big( \mathcal{B}, \cO(U(1)), \mathcal{A}_{0} \big)$ is a noncommutative principal $U(1)$-bundle as well.
\end{thm}

\begin{proof}
With the notation of Thm. \ref{quapriseq}, denoting  $\alpha_i^\gamma =\gamma^{-1}(\alpha_i) ,
\ \beta_i^\gamma=\gamma^{-1}(\beta_i) \ $, $\xi_i^\gamma=\gamma(\xi_i)$ and $\eta_i^\gamma=\gamma(\eta_i)$,
the collections $\{\xi_i^\gamma\}_{i=1}^N,\{\beta_i^\gamma\}_{i=1}^M\subset\B_1$ and $\{\eta_i^\gamma\}_{i=1}^N,\{\alpha_i^\gamma\}_{i=1}^M\subset\B_{-1}$ are such that:
$$
\sum\nolimits_{i=1}^N\xi_i^\gamma\star_\gamma\eta_i^\gamma=
\sum\nolimits_{i=1}^N\xi_i\eta_i=1
\;,\quad
\sum\nolimits_{i=1}^M\alpha_i^\gamma\star_\gamma\beta_i^\gamma=
\sum\nolimits_{i=1}^M\alpha_i\beta_i=1 \;.
$$
Hence the thesis, when applying Theorem~\ref{quapriseq}.
\end{proof}

\begin{rem}
There is an isomorphism of bimodules ${_{\gamma^n}}(\A_n){_{\gamma^{-n}}} \simeq {_{\gamma^{2n}}}(\A_n){_{\id}}$, 
implemented by the map $a\mapsto\gamma^n(a)$, for $a\in \A_n$. This map intertwines the deformed 
product $\star_\gamma$ with a new product
$$
a\star_{\gamma}' b=\gamma^{2n}(a)b \;,\qquad\text{for all} \;a\in\B_k,\,b\in\B_n ,
$$
and the undeformed involution with a new involution:
$$
a^{\dagger} = \gamma^{-2n}(a^*), \quad \mbox{ for all }  a\in \mathcal{B}_n .
$$
By construction $(\A,\star_\gamma)$ is isomorphic to $(\A,\star_\gamma')$ with deformed involution.

Suppose that $\A$ is dense in a graded $C^*$-algebra $A$ and $\gamma$ extends to a $C^*$-automorphism. Then, the completion $E_n$ of ${_{\gamma^{2n}}}(\A_n){_{\id}}$ becomes a self-Morita equivalence $A_0$-bimodule in the sense of Def.~\ref{de:sme} (with $\phi=\gamma^{2n}$), and the completion of $\B$ is the Pimsner algebra over $A_0$ associated to $E_{_1}={_{\gamma^{2}}}(\A_1){_{\id}}$. This is very similar to the construction in Example \ref{ex:3.3}. 
\end{rem} 

Examples of the above construction are noncommutative tori and $\theta$-deformed spheres and lens spaces, which we recall next.

\subsubsection{The noncommutative torus}
Being a crossed product, the noncommutative torus $C(\mathbb{T}^2_\theta) \simeq C(\IS^1) \rtimes_{\alpha} \IZ$ can be naturally seen as a Pimsner algebra over $C(\IS^1)$. The automorphism $\alpha$ of $C(\IS^1)$ is the one induced by the $\mathbb{Z}$-action generated by a rotation by $2\pi \ii \theta$ on $\IS^1$.
As a preparation for the examples of next section, let us see how it emerges from the deformed construction of the previous section.
 
Let $\A=\A(\IT^2)$ be the commutative unital $*$-algebra generated by
two unitary elements $u$ and $v$.
This is graded by assigning to $u,v$ degree $+1$ and to their adjoints degree $-1$.
The degree zero part is $\A_0\simeq\A(\IS^1)$, generated by the unitary $u^*v$.
Let $\theta\in\mathbb{R}$ and $\gamma$ be the graded $*$-automorphism given by
$$
\gamma_\theta(u)=e^{2 \pi \ii \theta}u \;,\qquad
\gamma_\theta(v)=v \;.
$$
From \eqref{eq:comm} we get
$$
u\star_{\gamma_\theta} v =e^{2\pi \ii \theta} v \star_{\gamma_\theta} u \;,
$$
plus the relations
$$
u\star_{\gamma_\theta} u^*=u^*\star_{\gamma_\theta} u=1 \;,\qquad
v\star_{\gamma_\theta} v^*=v^*\star_{\gamma_\theta} v=1 \;.
$$
Thus the deformed algebra $\B := (\A, \star_{\gamma_\theta}) = \A(\IT^2_\theta)$
is the algebra of the noncommutative torus.
 
\subsubsection{$\theta$-deformed spheres and lens spaces}
Let $\A=\A(\IS^{2n+1})$ be the commutative unital $*$-algebra generated by elements
$z_0,\ldots,z_n$ and their adjoints, with relation $\sum_{i=0}^n z_i^* z_i=1$.
This is graded by assigning to $z_0,\ldots,z_n$ degree $+1$ and to their adjoints degree $-1$.
For this grading the degree zero part is $\A_0\simeq\A(\mathbb{CP}^n)$. 
We denote by $\gamma$ the corresponding $\IS^1$-action on $\mathcal{A}$.

Any matrix $(u_{ij})\in U(n+1)$ defines a graded $*$-automorphism $\gamma$ by
$$
\gamma_u(z_i)=\sum\nolimits_{j=0}^nu_{ij}z_j \;,\qquad \;i=0,\ldots,n.
$$
Since a unitary matrix can be diagonalized by a unitary transformation, one
can assume that $(u_{ij})$ is diagonal. Denote $\lambda_{ij}=u_{ii}^2\bar u_{jj}^2$; from \eqref{eq:comm} one gets
$$
z_i\star_{\gamma_u} z_j  =\lambda_{ij}\, z_j\star_{\gamma_u} z_i \;,\qquad
z_i\star_{\gamma_u} z_j^* =\bar\lambda_{ij}\, z_j^*\star_{\gamma_u} z_i \;, \quad \mbox{ for all } i,j , 
$$
(and each $z_i$ is normal for the deformed product, since $\lambda_{ii}=1$), together with the conjugated relations, and a spherical relation
$$
\sum\nolimits_{i=0}^n z_i^* \star_\gamma z_i =1 \;.
$$
To use a more customary notation, consider the matrix $\Theta=(\theta_{ij})$ with components defined by $\lambda_{ij}=e^{2\pi\mathrm{i}\theta_{ij}}$. It is real (since $\lambda_{ij} \bar{\lambda}_{ij} = 1$), and antisymmetric (since $\bar\lambda_{ij}=\lambda_{ji}$).
We shall then denote by $\A(\IS^{2n+1}_{\Theta})$ the algebra $\A(\IS^{2n+1})$ with deformed product $\star_{\gamma}$. 

\begin{prop} 
The datum $ \big( \mathcal{A}(\IS^{2n+1}_{\Theta}), \cO(U(1)), \mathcal{A}(\CP^n) \big)$ is a noncommutative principal $U(1)$-bundle.
\end{prop}

With the same decomposition and notation as in \eqref{eq:Zd}, for any natural number $d$, consider the algebra
$$
\mathcal{A}(L^{2n+1}_{\theta}(d; \mv{1})) := \mathcal{A}(\IS^{2n+1}_{\Theta})^{\IZ_d} = 
\oplus_{n \in \IZ} \big(\mathcal{A}(\IS^{2n+1}_{\Theta})\big)_{dn} , 
$$  
which we think of as the coordinate algebra of the $\Theta$-deformed lens spaces. 
Clearly, for $d=1$ we get back the algebra $\mathcal{A}(\IS^{2n+1}_{\Theta})$.

\begin{prop} 
The datum $ \big( \mathcal{A}(L^{2n+1}_{\theta}(d; \mv{1})), \cO(U(1)), \mathcal{A}(\CP^n) \big)$ is a noncommutative principal $U(1)$-bundle.
\end{prop}

Let $C(\CP^{n})$, $C(\IS^{2n+1}_{\theta})$ and $C(L^{2n+1}_{\theta}(d; \mv{1}))$ the universal enveloping \mbox{\Cs algebras} for each of the coordinate algebra and by $E_1$ 
the completion of the spectral subspace $\B_1$. 
Since the circle action extended to $C(\IS^{2n+1}_{\theta})$ has large spectral subspaces, the $d$-th spectral subspace $E_{d}$ agrees with $(E_{1})^{\ots d}$. 
With $*$-homomorphism $\phi: A_{0} \rightarrow \cK(E)$ just left multiplication we have:

\begin{prop}
For all integers $d \geq 1$, the \mbox{\Cs algebra} $C(L^{2n+1}_{\theta}(d; \mv{1}))$ is a Pimsner algebra over $C(\CP^n)$ for the Hilbert bimodule $E_{d}$.
\end{prop}

\subsubsection*{Acknowledgment}
We thank Chiara Pagani for useful discussions and Bram Mesland for his 
feedback on an earlier version of the paper.
FD~was partially supported by UniNA and Compagnia di San
Paolo under the Program STAR 2013. GL was partially supported by the Italian Project Prin 2010-11 -- Operator Algebras, Noncommutative Geometry and Applications.

\end{document}